\documentclass[a4paper]{mathscan}
 
       \newtheorem{thm}{Theorem}[section] 

       \theoremstyle{definition} 
         
\newcommand{\rp}{ \Bbb R_+}
\newcommand{\crp}{\overline{\Bbb R}_+}
\newcommand{\Tr}{\operatorname{Tr}}

\begin{document}

\title [The sectorial projection defined from logarithms]
{The sectorial projection \\ defined from logarithms}

\author {Gerd Grubb}

\address
{Department of Mathematical Sciences, Copenhagen University,
Universitetsparken 5, DK-2100 Copenhagen, Denmark.
E-mail {\tt grubb\@math.ku.dk}}
\begin{abstract}
For a classical elliptic pseudodifferential operator 
$P$ of order $m>0$ on a closed
manifold $X$, such that the eigenvalues of the principal symbol
$p_m(x,\xi )$ have arguments in $\,]\theta ,\varphi [\,$ and
$\,]\varphi ,\theta +2\pi [\,$ ($\theta <\varphi <\theta +2\pi $), the
sectorial projection $\Pi _{\theta ,\varphi }(P)$ is defined
essentially as the integral of the resolvent along $e^{i\varphi
}\crp\cup e^{i\theta }\crp$. In a recent paper, Booss-Bavnbek, Chen,
Lesch and Zhu have pointed out that there is a flaw in several
published proofs that $\Pi _{\theta ,\varphi }(P)$ is a $\psi $do of
order 0; namely that $p_m(x,\xi )$ cannot in general be modified to allow
integration of $(p_m(x,\xi )-\lambda )^{-1}$ along   $e^{i\varphi
}\crp\cup e^{i\theta }\crp$ simultaneously for all $\xi $. We show
that the structure of $\Pi _{\theta ,\varphi }(P)$ as a $\psi $do of
order 0 can be deduced from the formula $\Pi _{\theta ,\varphi }(P)=\tfrac i{2\pi }( \log_\theta \! P -
\log_\varphi \! P )$ proved in an earlier work
(coauthored with Gaarde). In the analysis of $\log_\theta P$ one need
only modify $p_m(x,\xi )$ in a neighborhood of $e^{i\theta }\crp$;
this is known to be possible from Seeley's 1967 work on complex powers.
\end{abstract}

\maketitle

\section {Functions of an elliptic operator}\label{sec1}
 
Let $P$ be a classical elliptic pseudodifferential operator ($\psi $do)
of order $m>0$ acting in an $N$-dimensional hermitian vector bundle
$E$ over  a closed $n$-dimensional $C^\infty $ 
manifold $X$.

The construction of functions of $P$ was initiated by  Seeley, who in
\cite{S67} constructed and analysed the {\bf complex powers}
$P^s$ and showed that they are likewise classical $\psi $do's, under
the assumption that $P$ has one ray of minimal growth \linebreak$\{\lambda
=re^{i\theta }\mid r\in{\Bbb R}_+\}$, 
where  $(P-\lambda )^{-1}$ is well-defined and
is $O(\lambda ^{-1})$ for $\lambda \to\infty $. They are useful in
index theory for elliptic operators, and its generalizations, see also Atiyah and Bott \cite{AB67}. 
Greiner \cite{G71} defined the {\bf heat operator} $e^{-tP}$, when all
rays with argument in $\,]\pi /2-\delta , 3\pi /2+\delta [\,$ are rays
of minimal growth (for some $\delta >0$); it is likewise used in index
theory. 
The {\bf sectorial projection} $\Pi _{\theta
,\varphi }(P)$ 
is defined when
$P$ has two rays of
minimal growth $e^{i\theta }{\Bbb R}_+$ and $e^{i\varphi }{\Bbb R}_+$
($ \theta <\varphi <\theta + 2\pi $), as a projection whose range includes the generalized eigenspaces for
eigenvalues with argument in $\,]\theta ,\varphi [\,$, and whose
nullspace
contains the generalized eigenspaces for
eigenvalues with argument in $\,]\varphi ,\theta +2\pi [\,$.
Burak \cite{B70} studied it for $P$ equal to a realization $A_B$ of an elliptic
differential operator $A$ with a boundary condition $Bu=0$. A
special case, the {\bf positive eigenprojection} $\Pi _+(P)$ for a
selfadjoint differential or pseudodifferential operator, came into
focus with the works of Atiyah, Patodi and Singer \cite{APS75, APS76}
on index formulas for Dirac operators with boundary conditions;
here $\Pi _+(P)$ equals $(P+|P|)(2|P|)^{-1}$ (defined to be zero on the
nullspace of $P$), and is a $\psi $do of order 0 since
$|P|=(P^2)^\frac12$ is classical elliptic of order $m$ by \cite{S67}. 
Wodzicki \cite{W82, W84, W87}, in his studies of the spectral asymmetry, considered $\Pi _+(P)$, as well as more
general sectorial projections in cases where
$P$ has two rays of
minimal growth.
The {\bf logarithm} $\log_\theta P$ is defined when $e^{i\theta }{\Bbb R}_+$
is a ray of minimal growth, and arises e.g.\ as the derivative of
$P^s$ at $s=0$; it was analysed in detail by Okikiolu \cite{O95a, O95b} in
connection with determinant formulas.

The  sectorial projection can be defined on smooth functions by the formula
\begin{equation}
\Pi_{\theta,\varphi}(P)  = \tfrac{i}{2\pi} \int_{\Gamma_{\theta,\varphi}}
\lambda^{-1} P\, (P-\lambda)^{-1}  \, d\lambda, 
\label{1.1}
\end{equation}
where the integration goes along the sectorial contour 
\begin{equation}
\Gamma_{\theta,\varphi} = \{ r e^{i \varphi} \mid \infty > r > r_0
\} \cup \{ r_0 e^{i\omega} \mid \varphi \ge \omega \ge \theta
\} \cup \{ r e^{i \theta} \mid r_0 < r < \infty \};
\label{1.2}
\end{equation}
here $r_0$ is taken so small that 0 is the only possible eigenvalue in
$\{|\lambda |\le r_0\}$. 

Detailed studies of $\Pi_{\theta,\varphi}(P)$ were also made by Wojciechowski
\cite{Woj85} for applications to the spectral flow for first-order operators, by Nazaikinskii-Sternin-Shatalov-Schulze
\cite{NSSS98} for manifolds with singularities, and by Ponge
\cite{P06} who wanted to give a simplified proof of the results of
Wodzicki.
A  recent paper of Booss-Bavnbek, Chen, Lesch and Zhu \cite{BCLZ11}
gives an interesting observation, namely that there
is a flaw in the arguments of the latter three papers, where the
construction of $\Pi _{\theta ,\varphi }(P)$ is based on an application
of \eqref{1.1} to the terms in the symbol of $(P-\lambda )^{-1}$: When the
principal symbol $p_m(x,\xi )$ has eigenvalues both with arguments in
$\,]\theta ,\varphi [\,$ and in $\,]\varphi ,\theta +2\pi [\,$, {\it one cannot
obtain that  
$(p_m(x,\xi )-\lambda )^{-1}$ is
nonsingular along the curve $\Gamma _{\theta ,\varphi }$ 
 for all $\xi \in {\Bbb R}^n$}; there is a topological obstruction (see the detailed explanation in \cite{BCLZ11}). Therefore
a modified proof is needed.

The reader is referred to the paper of Booss-Bavnbek et al.\ for their strategy to
circumvent the mentioned difficulty. They show that 
$\Pi _{\theta ,\varphi }(P)$ is $H^s$-bounded, when $P$
has a homogeneous principal symbol and a lower-order part in $S^{m-1}_{1,0}$.
They use their estimates to show that the norm of 
$\Pi _{\theta ,\varphi }(P)$ in $H^s$-spaces 
depends continuously on $P$ in a certain symbol/operator topology
coarser than the full symbol topology.

We shall here show, when $P$ is classical,  that a very easy
proof of the fact that $\Pi _{\theta ,\varphi }(P)$ is a classical
$\psi $do of order 0 (in particular $H^s$-bounded) comes from the relation between the sectorial 
projection and
logarithms of $P$, as worked out in detail in Gaarde-Grubb \cite{GG08}.

\section {Preliminaries on the
logarithm of $P$}\label{sec2}

The present author's interest in the logarithm stems from reading the
paper of Scott \cite{S05}, where it was shown that $C_0(P)=-\frac1m\operatorname{res}(\log P)$; here $C_0(P)=\zeta
(P,0)+\dim\ker P$, where $\zeta (P,s)$ is the meromorphic extension of
$\Tr P^{-s}$, and the residue of $\log P$ is as defined in
Okikiolu \cite{O95b}.
Since $C_0(P)$ is also equal to the coefficient 
of $-\lambda ^{-1}$ in the expansion of the resolvent trace $\Tr(P-\lambda
)^{-1}$ (take $m>n$ for simplicity in this motivating
explanation), this coefficient is related in the same way to
$\operatorname{res}(\log P)$. We
wanted to give a direct proof of the last fact 
without having to calculate complex powers --- for the sake of a
generalization to boundary value problems where complex powers are
difficult to use. The outcome is explained in \cite{G05}, where the
point of departure is  a
simple key lemma (Lemma 1.2) that shows how the logarithm comes into 
the resolvent
trace calculations. This was used to show Scott's formula directly
from resolvent trace expansions, and the method was generalized to get
similar results for manifolds with boundary.
Related observations were used to deduce the results in \cite{GG08}
that we appeal to in Section \ref{sec3} below.

Assume in this section that $P$ is elliptic of order $m\in{\Bbb R}_+$,
having  a ray of minimal growth $e^{i\theta }{\Bbb R}_+$ for some 
$\theta \in [0,2\pi [\,$. This means that the principal symbol
$p_m(x,\xi )$, homogeneous of degree $m$ in $\xi $ for $|\xi |\ge 1$
and smooth in $(x,\xi )$, has no eigenvalues on $e^{i\theta }{\Bbb
R}_+$ when $|\xi | \ge 1$. Then
$(P-\lambda )^{-1}$ exists and is $O(\lambda ^{-1})$ for large
$\lambda $ on the ray, and since the hypotheses are valid also for
rays with argument close to $\theta $, one can assume that the ray is free of eigenvalues of $P$.

The principal
symbol of the resolvent is $q_{-m}(x,\xi ,\lambda )=(p_m(x,\xi
)-\lambda )^{-1}$ for $|\xi |\ge 1$, assumed to be extended in a
smooth way for $|\xi |\le 1$. The smoothing can be done for each $\lambda $
e.g.\ by multiplication of $(p^h_m(x,\xi)-\lambda )^{-1}$ (where
$p^h_m(x,\xi )$ denotes the strictly homogeneous symbol) by an
excision function  $\zeta (\xi )$ (a nonnegative
$C^\infty $ function that equals
1 for $|\xi |\ge 1$, 0  near $\xi
=0$). In some cases it suffices to modify
$p_m(x,\xi )$ itself for small $\xi $. 

The lower order terms $q_{-m-j}$ in
the symbol $q(x,\xi ,\lambda )$ of $(P-\lambda )^{-1}$ are defined
in local coordinates by recursive formulas known from \cite{S67}; they
are finite sums of terms with the structure
\begin{equation} 
r(x,\xi ,\lambda )=b_1q_{-m}^{\nu _1}b_2q_{-m}^{\nu _2}\cdots b_Mq_{-m}^{\nu _M}b_{M+1},\label{2.1}
\end{equation}
where the $b_k$ are homogeneous $\psi $do symbols independent of
$\lambda $, the $\nu _k$ are positive integers with sum $\ge 2$. 
$(P-\lambda )^{-1}$ has the full symbol
$q(x,\xi ,\lambda )\sim \sum_{j\ge 0}q_{-m-j}(x,\xi ,\lambda )$. The
terms $q_{-m-j}$ are quasi-homogeneous (homogeneous of degree
$-m-j$ in $(\xi ,|\lambda |^{\frac1m})$ on each ray).

Now $\log _\theta P$ can be defined on smooth functions by
\begin{equation}
\log _\theta P=\lim _{s\searrow 0}\tfrac i{2\pi }\int_{\mathcal C}\lambda_\theta 
^{-s}\log_\theta  \lambda \,(P-\lambda )^{-1} \,d\lambda.\label{2.2}
\end{equation}
Here $\lambda _\theta ^{-s}$ and $\log _\theta \lambda $ are taken with branch cut
 $e^{i\theta }\rp$, and $\mathcal C$
is a contour in ${\Bbb C}\setminus e^{i\theta }\crp$ going
around the nonzero spectrum of $P$ in the positive direction; for
precision we can take a Laurent loop
\begin{equation}
\mathcal C_\theta=\{ re^{i\theta }\mid \infty >r>r_0\}\cup\{ r_0e^{i\omega
}\mid \theta \ge \omega \ge \theta - 2\pi 
\}\cup\{re^{i (\theta-2\pi) }\mid r_0<r<\infty \},
\label{2.3}
\end{equation}
with $r_0$ so small that 0 is the only possible eigenvalue in
$\{|\lambda |\le r_0\}$.
As shown in Okikiolu \cite{O95a}, the symbol of $\log_\theta 
P$ is calculated in local coordinates from the resolvent symbol $q(x,\xi ,\lambda )$
by integration with $\log _\theta \lambda $ around the
spectrum of the principal symbol $p_m$. The
terms $q_{-m-j}$ contribute as
follows:
\begin{equation}\aligned
\tfrac i{2\pi }&\int_{\mathcal C_\theta (x,\xi )}\log_\theta  \lambda \,q_{-m}(x,\xi
,\lambda ) \,d\lambda=
\tfrac i{2\pi }\int_{\mathcal C_\theta (x,\xi )}\log _\theta  \lambda \,(p_m(x,\xi
)-\lambda )^{-1} \,d\lambda\\&=\log _\theta p_m(x,\xi )
=\log _\theta  ([\xi ]^m)+\log_\theta ([\xi ]^{-m}p_m(x,\xi ))=m\log[\xi ]+l_{\theta ,0}(x,\xi ),\\
\tfrac i{2\pi }&\int_{\mathcal C_\theta (x,\xi )}\log _\theta \lambda \,q_{-m-j}(x,\xi
,\lambda ) \,d\lambda=
l_{\theta ,-j}(x,\xi )\text{ for }j>0,\\
\endaligned\label{2.4}\end{equation}
where $\mathcal C_\theta (x,\xi )$ is a closed curve in $\Bbb C\setminus e^{i\theta }\crp$
around the spectrum of $p_m(x,\xi )$, and $[\xi ]$ stands for a smooth
positive function on ${\Bbb R}^n$ equal to $|\xi |$ for $|\xi |\ge 1$. It is a point here that the
Laurent loop used for $\log_\theta P$ is replaced by a closed curve (by
replacement of the rays outside a large $R$ by an arc with radius $R$)
and $ \lambda ^{-s}$ is replaced by its limit 1, since the
spectrum is bounded at each $(x,\xi )$. If, more generally, $\theta =\theta _0+2\pi k$
with $\theta _0\in [0,2\pi [\,$ and $k$ integer, then $l_{\theta ,0}$
contains an additional constant $2\pi ik$.

Each $l_{\theta ,-j}$ is homogeneous in
$\xi $ of degree $-j$ for $|\xi |\ge 1$; for $j=0$ it follows since
$[\xi ]^{-m}p_m(x,\xi )$ is so, and for $j\ge 1$ it is seen e.g.\ as
follows (where we set $\lambda =t^m\varrho $, $t\ge 1$):
\begin{equation}\aligned
l_{\theta ,-j}(x,t\xi )&=\tfrac i{2\pi }\int_{\mathcal C_\theta (x,\xi )}\log _\theta \lambda \,q_{-m-j}(x,t\xi
,\lambda ) \,d\lambda\\
&= \tfrac i{2\pi }\int_{t^{-m}\mathcal C_\theta (x,\xi )}(\log _\theta 
\varrho+m\log t)t^{-m-j} \,q_{-m-j}(x,\xi
,\varrho  ) \,t^md\varrho \\
&= t^{-j}l_{\theta ,-j}(x,\xi )+mt^{-j}\log t\tfrac i{2\pi }\int_{t^{-m}\mathcal C_\theta (x,\xi )} \,q_{-m-j}(x,\xi
,\varrho  ) \,d\varrho ,
\endaligned\label{2.5}\end{equation} 
where the last term is zero since $q_{-m-j}$ is $O(|\varrho |^{-2})$
for $|\varrho |\to\infty $ when $j>0$.

In the proof that the full symbol of $\log_\theta  P$ is 
\begin{equation}
m\log[\xi ]+l_\theta (x,\xi ),\quad l_\theta (x,\xi )\sim {\sum}_{j\ge
0}l_{\theta ,-j}(x,\xi ),\label{2.6}
\end{equation}
one uses the observation by Seeley \cite{S67} that the symbol $p_m(x,\xi )$
can be modified smoothly near $\xi =0$ in such a
way that $p_m(x,\xi )-\lambda $ is invertible for all $\lambda $ in a
keyhole region $\{\lambda \in{\Bbb C}\mid \arg\lambda \in \,]\theta
-\delta  ,\theta +\delta  [\,\text{ or }|\lambda |<r\}$, all
$(x,\xi )$, when $r$ and $\delta  $ are sufficiently small positive numbers.

\section {The relation between the sectorial projection and
logarithms}\label{sec3}

For a general closed, densely defined operator $A$ in a Hilbert space $H$,
with compact resolvent and two rays $e^{i\theta }{\Bbb R}_+$ and
$e^{i\varphi }{\Bbb R}_+$ in the resolvent set, where \linebreak
$\|(A-\lambda )^{-1}\|$ is $O(\lambda ^{-1})$ for $\lambda \to\infty $
on the rays, one defines $\Pi_{\theta,\varphi}(A) $ on $D(A)$, to
begin with, by
\begin{equation}
\Pi_{\theta,\varphi}(A) x = \tfrac{i}{2\pi} \int_{\Gamma_{\theta,\varphi}}
\lambda^{-1} A\, (A-\lambda)^{-1} \, x \, d\lambda, \quad x\in D(A);
\label{3.1}
\end{equation}
here the integration goes along the sectorial contour \eqref{1.2}. 
If the hereby defined operator $\Pi_{\theta,\varphi}(A)$ is bounded in
$H$-norm, we extend it by continuity to $H$.

Similarly, if $A$ has compact resolvent and one ray $e^{i\theta }{\Bbb
R}_+$
 in the resolvent set, where 
$\|(A-\lambda )^{-1}\|$ is $O(\lambda ^{-1})$ for $\lambda \to\infty
$ on the ray, one can define $\log_\theta A$ by the formula
\begin{equation}
\log _\theta Ax=\lim _{s\searrow 0}\tfrac i{2\pi }\int_{\mathcal C}\lambda_\theta 
^{-s}\log_\theta  \lambda \,(A-\lambda )^{-1}\,x \,d\lambda, \quad
x\in D(A).
\label{3.2}
\end{equation}

The results in the following theorem were shown in \cite{GG08} (Lemma
4.3 and Prop. 4.4).

\begin{thm}\label{Theorem 3.1}

$1^\circ$ Let $f(\lambda)$ be a continuous (possibly vector-valued) function on
the ``punctuated double keyhole region'' 
\begin{equation}
V_{r_0,\delta} = \{ \lambda \in \Bbb C \mid |\lambda| < 2r_0 \text{ or
} |\arg \lambda - \theta| < \delta \text{ or } |\arg\lambda -
\varphi|< \delta \} \setminus \{ 0 \},
\label{3.3}
\end{equation}
such that $f(\lambda)$ is $O(\lambda^{-1-\varepsilon})$ for $|\lambda|
\to \infty$ in $V_{r_0,\delta}$. Then
\begin{equation}
\int_{\mathcal C_{\theta}} \log_\theta \! \lambda \,
f(\lambda) \,d\lambda - \int_{\mathcal C_{\varphi}} \log_\varphi \! \lambda \,
f(\lambda)\, d\lambda = -2\pi i \int_{\Gamma_{\theta,\varphi}} f(\lambda)
\,d\lambda.
\label{3.4}
\end{equation}

$2^\circ$
For $x \in D(A)$,
\begin{equation}
\log_\theta \! A \, x - \log_\varphi \! A \, x =
\int_{\Gamma_{\theta,\varphi}} \lambda^{-1} A (A-\lambda)^{-1} \, x \,
d\lambda = -2\pi i \; \Pi_{\theta,\varphi}(A) \, x.
\label{3.5}
\end{equation}
When $\Pi_{\theta,\varphi}(A)$ is bounded, so is
$\log_\theta \! A  - \log_\varphi \! A  $ (and vice versa), and then 
\begin{equation}
\Pi_{\theta,\varphi}(A)=\tfrac i{2\pi }( \log_\theta \! A -
\log_\varphi \! A ).
\label{3.6}
\end{equation}
\end{thm}

We now assume that $P$ is elliptic of order $m\in{\Bbb R}_+$, having  two rays of minimal growth $e^{i\theta }{\Bbb R}_+$ and $e^{i\varphi }{\Bbb R}_+$
(with $\theta \in [0,2\pi [\,$, $\theta <\varphi <\varphi + 2\pi
$). Then
$(P-\lambda )^{-1}$ exists and is $O(\lambda ^{-1})$ for large
$\lambda $ on the rays, and we can
assume that the rays are free of eigenvalues of $P$.
The considerations in Theorem 3.1 will be applied to $P$, 
entering as a closed unbounded
operator in $H=L_2(X)$ with domain $D(P)=H^m(X)$.

\begin{thm}\label{Theorem 3.2} $\Pi _{\theta ,\varphi }(P)$ equals $\tfrac i{2\pi }( \log_\theta \! P -
\log_\varphi \! P )$, and is a classical $\psi $do of
order $\le 0$. It has the symbol
\begin{equation}
\pi _{\theta ,\varphi }(x,\xi )=\tfrac i{2\pi }(l_{\theta }(x,\xi
)-l_{\varphi }(x,\xi )),\label{3.7} 
\end{equation}
in local coordinates.
\end{thm}
\begin{proof} As recalled in Section \ref{sec2}, 
$\log_\theta P$ and $\log_\varphi P$ are log-polyhomogeneous $\psi
$do's with symbols as described in \eqref{2.6}. Then 
$ \log_\theta \! P-
\log_\varphi \! P $ has the symbol
\begin{equation}
l_{\theta }(x,\xi )-l_{\varphi }(x,\xi ),\label{3.8}
\end{equation}
where the log-terms $m\log[\xi ]$ have cancelled out. Hence it is a
classical $\psi $do of order $\le 0$; in particular it is bounded on
$L_2(X)$. By Theorem \ref{Theorem 3.1} $2^\circ$, $\Pi_{\theta,\varphi}(P)$ is then also
bounded on $L_2(X)$; and it equals $\tfrac i{2\pi }( \log_\theta \! P -
\log_\varphi \! P )$ and is a classical
$\psi $do on $X$ of order $\le 0$ whose symbol is found as the corresponding
linear combination of the symbols of the logarithms,
namely \eqref{3.7}.\qed
\end{proof}

We can also show that the terms in the symbol of $\Pi _{\theta
,\varphi }(P)$ have the expected form as integrals of terms in the
resolvent symbol:

\begin{thm}\label{Theorem 3.3} In local coordinates, the symbol $\pi _{\theta ,\varphi }(x,\xi )\sim\sum_{j\ge 0}\pi
_{\theta ,\varphi ,-j}(x,\xi )$ of $\Pi _{\theta ,\varphi }(P)$
satisfies, 
 for each $x$, each $|\xi |\ge 1$: 
\begin{equation}
\pi _{\theta ,\varphi ,-j}(x,\xi )=\tfrac i{2\pi }(l_{\theta ,-j}(x,\xi )-l_{\varphi ,-j}(x,\xi ))=  \tfrac{i}{2\pi} \int_{\mathcal
C_{\theta,\varphi}(x,\xi)} q_{-m-j}(x,\xi,\lambda) \, d\lambda,\label{3.9}
\end{equation}
for all $j$. Here $\mathcal C_{\theta,\varphi}(x,\xi)$ is a closed curve
in the open sector
$\Lambda_{\theta,\varphi}=$ \linebreak $\{\lambda \in {\Bbb C}\mid \theta <\arg
\lambda <\varphi \}$ going in the positive direction around the
spectrum of $p_m(x,\xi)$ lying in that sector. 
\end{thm}

\begin{proof} The first equality follows immediately from \eqref{3.7}.
For the second equality we use \eqref{2.4}. For $j=0$,
  we obtain the formula by applying Theorem \ref{Theorem 3.1} $2^\circ$ 
to the bounded operator
$p_m(x,\xi )$ in  ${\Bbb C}^N$:
\begin{align*}
l_{\theta ,0}-l_{\varphi ,0}&=\log_\theta p_m-\log_\varphi
p_m=\int_{\Gamma_{\theta,\varphi}} \lambda^{-1} p_m (p_m-\lambda)^{-1}
\,d\lambda \\
&=\int_{\mathcal C} \lambda^{-1} p_m (p_m-\lambda)^{-1}
\,d\lambda 
=\int_{\mathcal C} (\lambda^{-1} + (p_m-\lambda)^{-1}
) \,d\lambda \\
&=\int_{\mathcal C} (p_m-\lambda)^{-1} \,d\lambda ,
\end{align*}
where the curve $\Gamma _{\theta
,\varphi }$ could be replaced by a closed curve $\mathcal C=\mathcal C_{\theta
,\varphi }(x,\xi )$ in $\Lambda _{\theta ,\varphi }$,
since the integrand was $O(\lambda ^{-2})$ and the spectrum of
$p_m(x,\xi )$ 
in 
$\Lambda _{\theta ,\varphi } $ is a finite set of points.
 For $j\ge 1$, we obtain the formula by application of 
Theorem \ref{Theorem 3.1} $1^\circ$ with
 $f(\lambda )$ equal to the $j$'th term $q_{-m-j}(x,\xi ,\lambda )$ in
 the resolvent symbol (recall the structure as a sum of terms \eqref{2.1});
it is $O(\lambda ^{-2})$, allowing reduction to a closed curve $\mathcal C$.   
\qed
\end{proof}

In relation to the problem raised in \cite{BCLZ11}, we note that the
calculations in Theorem \ref{Theorem 3.3} take place at individual points $(x,\xi
)$, where there is no problem with singularities on the curve $\Gamma
_{\theta ,\varphi }$. Calculations global in $\xi $ are only
performed in the constructions of the logarithms, where the argument
of Seeley \cite{S67} is valid.

In \cite{GG08} we relied on the account of Ponge \cite{P06} referring
to five works of Wodzicki from the 80's (two in Russian), for the knowledge that $\Pi _{\theta ,\varphi }(P)$
is a zero-order $\psi $do. Although Ponge's own formulation of a
proof has the flaw pointed out in \cite{BCLZ11}, we see no reason to
doubt the original statement, which is further supported by the formula $ P_{\theta}^s-P_{\varphi}^s=(1-e^{2i\pi s}) \Pi _{\theta,\varphi}(P) P_{\theta}^s
$ in \cite{P06} Sect.\ 4, ascribed to Wodzicki.

 At any rate, it seems to be
useful that the present paper gives an independent proof which avoids the
mentioned pitfall, and is based directly on resolvents and logarithms.

The formula $\Pi _{\theta ,\varphi }(P)=\tfrac i{2\pi }( \log_\theta \! P -
\log_\varphi \! P )$  allows a direct application of the results in Okikiolu
\cite{O95a}, Sect.\ 4, to show that the norm in $H^s$-spaces 
depends continuously and even smoothly on the symbol of $P$, in
dependence on a parameter $t$ in an open subset $T$ of ${\Bbb R}^d$.
 As mentioned earlier,
\cite{BCLZ11} shows the continuity in terms of a certain
symbol/operator topology on $P$.

\end{document}